\documentclass[a4paper,10pt]{amsart}
\usepackage[utf8]{inputenc}
\usepackage{amsmath,amsthm,mathtools,amssymb, dsfont,bm}
\usepackage{hyperref,xparse}
\usepackage{a4wide}
\usepackage{multicol,breqn,enumerate}
\usepackage{pictexwd,dcpic,aliascnt}
\usepackage{tikz}
\usetikzlibrary{arrows,positioning}

\newtheorem{theorem}{Theorem}[section]

\newaliascnt{lemma}{theorem}
\newtheorem{lemma}[lemma]{Lemma}
\aliascntresetthe{lemma}

\newaliascnt{fact}{theorem}
\newtheorem{fact}[fact]{Fact}
\aliascntresetthe{fact}

\newaliascnt{corollary}{theorem}
\newtheorem{corollary}[corollary]{Corollary}
\aliascntresetthe{corollary}

\newaliascnt{proposition}{theorem}
\newtheorem{proposition}[proposition]{Proposition}
\aliascntresetthe{proposition}

\newaliascnt{df}{theorem}
\theoremstyle{definition}\newtheorem{df}[df]{Definition}
\aliascntresetthe{df}

\newaliascnt{remark}{theorem}
\theoremstyle{remark}\newtheorem{remark}[remark]{Remark}
\aliascntresetthe{remark}

\numberwithin{equation}{section}

\DeclareMathOperator{\Pol}{\mathcal{O}}
\DeclareMathOperator{\id}{id}
\DeclareMathOperator{\perm}{perm}
\DeclareMathOperator{\sgn}{sgn}
\DeclareMathOperator{\spn}{span_{\mathbb{C}}}

\newcommand{\GG}{\mathbb{G}}
\newcommand{\HH}{\mathbb{H}}

\newcommand{\bigslant}[2]{{\raisebox{.2em}{$#1$}\left/\raisebox{-.2em}{$#2$}\right.}}

\newcommand{\CA}{$C^{\ast}$-algebra}
\newcommand{\HA}{Hopf ${}^{\ast}$-algebra}
\newcommand{\nc}{{\text{---}\scriptscriptstyle{\|\cdot\|}}}

\newcommand\comp{\!\circ\!}

\newcommand{\mult}[2]{\mathsf{M}({#1}\otimes{#2})}

\title[Remark on Hopf images in $S_n^+$]{Remark on Hopf images in quantum permutation groups $S_n^+$}
\author{Pawe{\l} J{\'o}ziak}
\address{Institute of Mathematics of the Polish Academy of Sciences,
ul.~\'Sniadeckich 8, 00--656 Warszawa, Poland and  Institute of mathematics, University of Wroc{\l}aw, pl. Grunwaldzki 2/4, 50--384 Wroc{\l}aw, Poland}  \email{pjoziak@impan.pl}


\begin{document}

\maketitle

\begin{abstract}
Motivated by a question of A.~Skalski and P.M.~So{\l}tan about inner faithfulness of the S.~Curran's map, we revisit the results and techniques of T.~Banica and J.~Bichon's Crelle paper and study some group-theoretic properties of the quantum permutation group on $4$ points. This enables us not only to answer the aforementioned question in positive in case $n=4, k=2$, but also to classify the automorphisms of $S_4^+$, describe all the embeddings $O_{-1}(2)\subset S_4^+$ and show that all the copies of $O_{-1}(2)$ inside $S_4^+$ are conjugate. We then use these results to show that the criterion we applied to answer the aforementioned question does not admit converse.
\end{abstract}

\section*{Introduction}
Let $\GG$ be a compact quantum group (in the sense of Woronowicz, but throughout the note we will not need any of the analytic features of the associated Hopf-\CA), let $\Pol(\GG)$ be its associated coordinate ring and assume $\beta\colon\Pol(\GG)\to\mathcal{B}$ is a ${}^*$-representation of $\Pol(\GG)$ as a ${}^*$-algebra in some ${}^*$-algebra. Via abstract Gelfand-Naimark duality, such a maps corresponds to a map $\hat{\beta}\colon\mathbb{X}\to\GG$ and it is natural to ask what is the smallest quantum subgroup containing $\hat{\beta}(\mathbb{X})$, or -- in other words -- what the quantum subgroup generated by $\hat{\beta}(\mathbb{X})\subset\GG$ is. The answer to this type of questions was studied earlier in \cite{Ban14,BB10,BCV,SS16} in the case of compact quantum groups and later extended to locally compact quantum groups in \cite{PJphd,JKS16}.

The concept of a subgroup is central to treating quantum groups from the group-theoretic perspective and many efforts were made to provide accurate descriptions of various aspects of this concept, as well as providing some nontrivial examples, see, e.g., \cite{BB09,BY14,DKSS,Pod95} and many others. Throughout this manuscript, we deal with subgroups of the quantum permutation groups, first introduced by Wang in \cite{Wang98}. It was observed in \cite{KS09} that quantum permutations can be used to study distributional symmetries of infinite sequences of non-commutative random variables that are identically distributed and free modulo the tail algebra, thus extending the classical de Finetti's theorem to the quantum/free realm.

Another extension of de Finetti's theorem was given by Ryll-Nardzewski: he observed that instead of invariance of joint distributions under permutations of random variables it is enough to consider subsequences and compare these type of joint distributions to obtain the same conclusion. What this theorem really boils down to is the fact that one can canonically treat the set $I_{k,n}$ of increasing sequences (of indices) as subset of all permutations $S_n$, and this subset is big enough to generate the whole symmetric group: $\langle I_{k,n}\rangle=S_n$, unless $k=0$ or $k=n$.

This viewpoint was utilized in \cite{Cur11} by Curran to extend theorem of Ryll-Nardzewski to the quantum case: he introduced the quantum space of quantum increasing sequences $I^+_{k,n}$ and defined how to canonically extend the quantum increasing sequence to a quantum permutation in $S_n^+$. The analytical properties of the \CA\ $C(I^+_{k,n})$ were strong enough to provide an extension of Ryll-Nardzewski to the quantum/free case. However, these results did not say anything about the subgroup of quantum permutation group that is generated by quantum increasing sequences.

If the analogy with the classical world is complete, one would expect that in fact $\overline{\langle I^+_{k,n}\rangle}=S_n^+$ for all $n$ and $k\neq0,n$. This was ruled out already in \cite{SS16}, where it was observed that $\overline{\langle I^+_{k,n}\rangle}=S_n$ whenever $k=1,n-1$. The second best thing one could hope for is that $\overline{\langle I^+_{k,n}\rangle}=S_n^+$ for at least one $k\in\{2,\ldots,n-2\}$, as this would explain the results of Curran in a more group-theoretic manner. In general, \cite[Question 7.3]{SS16} asks for the complete description of all $\overline{\langle I^+_{k,n}\rangle}$ and emphasizes the case $n=4$ and $k=2$ as the first non-trivial case to study. We give a positive answer in this case using the following lower bound criterion for the Hopf image: assume $\beta\colon C^u(\GG)\to \mathsf{B}$ is a morphism and assume $X$ is the set of all characters of $\mathsf{B}$. Denoting by $\HH$ the Hopf image of $\beta$ we have that $\overline{\langle X\rangle}\subset\HH$. The \CA\ language is mainly used for convenience and it is straightforward to adapt this criterion to the purely algebraic situation. It should be noted that the inclusion can be proper for some analytical reasons, but it can be shown that even when restricting to the setting of topological generation of quantum groups in the spirit of \cite{BCV}, such an inclusion can still be proper. 

In the course of analyzing the impossibility of getting strict equality in the aforementioned criterion even in the analytically best-behaved case of coamenable compact quantum group of Kac type, we study some group-theoretic properties of the quantum permutation group $S_4^+$. Namely, we classify all Hopf automorphisms of $C(S_4^+)$ and show that there are three copies of $O_{-1}(2)$ appearing as quantum subgroups of $S_4^+$, and that they are conjugate. 

The manuscript is organized as follows. \autoref{sec:cqg-hi-crit} serves mainly as preliminaries needed to settle the notation for compact quantum groups (\autoref{sec:cqg}), Hopf images (\autoref{sec:hopfimage}) and quantum permutations groups together with quantum increasing sequences (\autoref{sec:qpg-qis}). However, the main criterion is also contained there as \autoref{thm:criterion}, as well as the answer to \cite[Question 7.3]{SS16}, as \autoref{thm:hopfimage}. In \autoref{sec:gtprop} we turn to studying group-theoretic properties of $S_4^+$. We introduce the objects we need in \autoref{sec:objects} and later we revise the technique of cocycle twists in \autoref{sec:twistingsgeneral}. We also introduce the concept of characteristic subgroups in \autoref{sec:characteristic} in the context of compact quantum groups. In \autoref{sec:twistingapplied} we recall how the technique of cocycle-deformation is applied to $S_4^+$ and use the results of \autoref{sec:characteristic} to classify quantum automorphisms of $S_4^+$. In \autoref{sec:embeddings1} we classify embeddings $O_{-1}(2)\subset S_4^+$ and use them to show in \autoref{sec:conclusions} that the inclusion in our criterion, \autoref{thm:criterion}, can be proper even in the analytically best-behaved setting. We also gather there also some other consequences of our results.

\section{Compact quantum groups, Hopf image and Criterion}\label{sec:cqg-hi-crit}
Throughout the manuscript, we will use tensor products of different structures, mainly ${}^{\ast}$-algebras (the algebraic tensor product) and $C^{\ast}$-algebras (the minimal tensor product). It will be denoted using the same symbol $\otimes$, as this should be clear from the context which tensor product procedure is evoked at the time. The \CA\ formalism is used only for convenience, as all the results rely only on their algebraic features. The $C^*$-algebra of compact operators on a Hilbert space $\mathcal{H}$ is denoted $\mathsf{K}(\mathcal{H})$ and for a $C^*$-algebra $\mathsf{A}$ we denote by $\mathsf{M}(\mathsf{A})$ its multiplier algebra.
\subsection{Compact quantum groups}\label{sec:cqg}
In this section we recall the basic definitions from the theory of compact quantum groups. We stick to the formalism established in \cite{Wor87, Wor95}. A unital \CA\ $\mathsf{A}$ endowed with a ${}^{\ast}$-homomorphism $\Delta\colon \mathsf{A}\to \mathsf{A}\otimes \mathsf{A}$ satisfying the coassociativity condition: $(\Delta\otimes \id)\comp\Delta=(\id\otimes\Delta)\comp\Delta$ is called a \emph{Woronowicz algebra}, if the \emph{cancellation laws} hold:
\[\spn^{\nc}\big((\mathds{1}\otimes \mathsf{A})\Delta(\mathsf{A})\big)=\mathsf{A}\otimes\mathsf{A} = \spn^{\nc}\big((\mathsf{A}\otimes \mathds{1})\Delta(\mathsf{A})\big)\]
where $\spn^{\nc}$ denotes the norm closure of the linear span. 

Such an algebra corresponds to a \emph{compact quantum group} $\GG$ via the identification $\mathsf{A}=C(\GG)$, the algebra of continuous functions on $\GG$. It can be endowed with a unique state $h\in \mathsf{A}^{\ast}$, called the \emph{Haar state}, which is left and right invariant:
\[(\id\otimes h)\comp\Delta=(h\otimes\id)\comp\Delta=h(\cdot)\mathds{1}.\]
$\mathsf{A}$ contains a unique dense Hopf ${}^{\ast}$-subalgebra $\Pol(\GG)$ (i.e.~the coproduct $\Delta$ restricts to $\Pol(\GG)$); it is spanned by matrix coefficients of unitary representations of $\GG$. $\Pol(\GG)$ can have, a priori, a plethora of different $C^{\ast}$-norms: the norm coming from GNS-representation of the Haar state ($C^r(\GG)=\overline{\Pol(\GG)}\subseteq\mathsf{B}(L^2(\GG))$), the norm of $\mathsf{A}$ and the universal $C^{\ast}$-norm need not coincide. For further discussion on this topic, see e.g. \cite{KS12}. In any case, there are always quotient maps
\[C^u(\GG)\to C(\GG)\to C^r(\GG)\]
where $C(\GG)$ denotes a general $C^{\ast}$-completion. If the quotient map $\Lambda\colon C^u(\GG)\to C^r(\GG)$ is injective, we call $\GG$ coamenable and declare that $C^u(\GG)=C^r(\GG)$ and $\Lambda=\id$. In this note we mainly deal with coamenable compact quantum groups and use the symbol $C(\GG)$ to describe that \CA.

The most studied examples are the \emph{compact matrix quantum groups}: $\GG$ is a compact matrix quantum group if the Woronowicz algebra $C^u(\GG)$ can be endowed with a \emph{fundamental} corepresentation $u\in M_n(C^u(\GG))=\mult{\mathsf{K}(\mathbb{C}^n)}{C^u(\GG)}$: denoting $u_{i,j}=(\langle e_i|\cdot|e_j\rangle\otimes\id)u$ for a fixed basis $(e_i)_{1\leq i\leq n}\subset \mathbb{C}^n$, orthonormal with respect to an inner product $\langle\cdot|\cdot\rangle$, we ask for: 
\[\Delta(u_{i,j})=\sum_{k=1}^n u_{i,k}\otimes u_{k,j}\]
and \[\langle \{u_{i,j}:1\leq i,j\leq n\}\rangle=\Pol(\GG)\]
where $\langle X\rangle$ denotes the ${}^{\ast}$-algebra generated by elements of $X$ (note we used the symbol $\langle\quad\rangle$ also to denote the subgroup generated by a given subset, this shall cause no confusion). 

Any compact quantum group $\GG$ has its maximal classical subgroup $Gr(\widehat{\GG})$, also called the group of characters of $\GG$ or the intrinsic subgroup of $\widehat{\GG}$. It is given as follows: consider the universal enveloping $C^{\ast}$-algebra $C^u(\GG)$ and the commutator ideal of it, i.e.~the ideal generated by $\{xy-yx\colon x,y\in C^u(\GG)\}$, call this ideal $I$. Then the quotient map \[q_{\GG}\colon C^u(\GG)\to\bigslant{C^u(\GG)}{I}=:C(Gr(\widehat{\GG}))\]
identifies the spectrum of the (commutative) $C^{\ast}$-algebra $\bigslant{C^u(\GG)}{I}$, denoted $Gr(\widehat{\GG})$, with a closed (quantum) subgroup of $\GG$. The commutativity of the $C^{\ast}$-algebra  $C(Gr(\widehat{\GG}))$ ensures us that it is the only possible completion of $\Pol(Gr(\widehat{\GG}))$, so we drop the ${\cdot}^u$ decoration. A thorough description of the group of characters of a given (locally) compact quantum group $\GG$ can be found in \cite{KN13}
\subsection{Hopf image}\label{sec:hopfimage}
The Hopf image construction, studied in detail in the case of compact quantum groups in \cite{BB10, SS16} and in the case of locally compact quantum groups in \cite{JKS16}, is concerned with the following situation. Consider a (closed) subset in a (locally) compact group $X\subseteq G$. We are looking for the closed subgroup of $G$, say $H$, which is generated by the set $X$, i.e.~$\overline{\langle X\rangle}=H$. Under Gelfand-Naimark duality, this corresponds to finding the final/terminal object in the category, whose objects are defined with the aid of the following diagram:
 \begin{center} \begin{tikzpicture}
  [bend angle=36,scale=2,auto,
pre/.style={<<-,shorten <=1pt,semithick},
post/.style={->>,shorten >=1pt,semithick}]
\node (G) at (-0.7,0.7) {$C(G)$};
\node (X) at (0.7,0.7) {$C(X)$}
edge [pre] node[auto,swap] {$\beta$} (G);
\node (H) at (0,0) {$C(H)$}
edge [pre] node[auto,swap] {$\pi$} (G)
edge [post] node[auto] {$\tilde{\beta}$} (X);
\end{tikzpicture}\end{center} 
In the above diagram, $\beta$ is the Gelfand dual to the embedding $X\subset G$, $\beta'$ is the Gelfand dual to the embedding $X\subset H$ and $\pi$ is the Gelfand dual to the embedding $H\subset G$. The objects of the aforementioned category are triples consisting of commutative Woronowicz algebras $C(H)$ and maps $\pi$, $\tilde{\beta}$ so that $\pi$ intertwines the coproduct and $\tilde{\beta}\comp\pi=\beta$. In terms of spectra of these \CA s, this category consists of closed subgroups of $G$ containing the set $X$, and we are looking for a minimal one.

Dropping commutativity enables us to discuss \emph{closed quantum subgroup of $\GG$ generated by a map $\beta$}: the Hopf image of a ${}^{\ast}$-homomorphism $\beta\colon C^u(\GG)\to\mathsf{B}$ is the final object of the category, whose objects are triples consisting of Woronowicz algebras $C^u(\HH)$, Hopf ${}^{\ast}$-homomorphism $\pi\colon C^u(\GG)\to C^u(\HH)$ and a ${}^*$-homomorphism $\tilde{\beta}\colon C^u(\HH)\to \mathsf{B}$ such that the following diagram commutes
\begin{center} \begin{tikzpicture}
  [bend angle=36,scale=2,auto,
pre/.style={<<-,shorten <=1pt,semithick},
post/.style={->>,shorten >=1pt,semithick}]
\node (G) at (-0.7,0.7) {$C^u(\GG)$};
\node (X) at (0.7,0.7) {$\mathsf{B}$}
edge [pre] node[auto,swap] {$\beta$} (G);
\node (H) at (0,0) {$C^u(\HH)$}
edge [pre] node[auto,swap] {$\pi$} (G)
edge [post] node[auto] {$\tilde{\beta}$} (X);
\end{tikzpicture}\end{center} 
It should be stressed that the universal $C^{\ast}$-completions have the biggest amount of possible ${}^{\ast}$-homomorphisms $\beta$, so looking for the Hopf images of the maps defined on the universal $C^{\ast}$-completions is the only reasonable from the group-theoretic perspective. For instance, the compact quantum group $\widehat{\mathbb{F}_2}$, the dual to the free group on two generators, has simple reduced completion, and thus in the reduced world the only possible map is inclusion, and if we were to speak of Hopf images treating different $C^{\ast}$-algebras as different quantum groups, the Hopf image of any morphism from $C^{\ast}_r(\mathbb{F}_2) $ is always the whole of $\widehat{\mathbb{F}_2}$ -- even the trivial group is excluded!

Let $\HH_1,\HH_2\subset\GG$ be two closed quantum subgroups (identified via $\pi_i\colon C^u(\GG)\to C^u(\HH)$). Then we say that $\GG$ is topologically generated by $\HH_1$ and $\HH_2$, and write $\GG=\overline{\langle\HH_1,\HH_2\rangle}$, if the Hopf image of the either of the maps $(\pi_1\otimes\pi_2)\comp\Delta$ or $\pi_1\oplus\pi_2$, is the whole $C^u(\GG)$. This notion has several equivalent descriptions, see \cite[Proposition 3.5]{BCV} and \cite[Section 3]{JKS16}.

We now give a criterion showing ``how big'' the Hopf image of a ${}^{\ast}$-homomorphism $\beta\colon C^u(\GG)\to\mathsf{B}$ has to be. The content of the consecutive observation is best seen in the following diagram.
\begin{center} \begin{tikzpicture}
  [bend angle=36,scale=2,auto,
pre/.style={<<-,shorten <=1pt,semithick},
post/.style={->>,shorten >=1pt,semithick}]
\node (G) at (-1.5,1.5) {$C^u(\GG)$};
\node (X) at (1.5,1.5) {$\mathsf{B}$}
edge [pre] node[auto,swap] {$\beta$} (G);
\node (H) at (0,0.6) {$C^u(\HH)$}
edge [pre] node[auto,swap] {$\pi$} (G)
edge [post] node[auto] {$\tilde{\beta}$} (X);
\node (Hclass) at (0,-0.1) {$C(Gr(\widehat{\HH}))$}
edge [pre] node[auto,swap] {$q_{\HH}$} (H);
\node (Gclass) at (-1.5,-1) {$C(Gr(\widehat{\GG}))$}
edge [pre] node[auto,swap] {$q_{\GG}$} (G)
edge [post] node[auto] {$p$} (Hclass);
\node (Xclass) at (1.5,-1) {$C(\sigma(\mathsf{B}))$}
edge [pre] node[auto,swap] {$b$} (Gclass)
edge [pre] node[auto,swap] {$\tilde{b}$} (Hclass)
edge [pre] node[auto,swap] {$q_{\mathsf{B}}$} (X);
\end{tikzpicture}\end{center} 
Here $q_{\GG}\colon C^u(\GG)\to C(Gr(\widehat{\GG}))$ is the canonical embedding of the group of characters $Gr(\widehat{\GG})\subset \GG$ (likewise for $\HH$); $C(\sigma(\mathsf{B}))$ is the quotient of $\mathsf{B}$ by the commutator ideal and $\sigma(\mathsf{B})$ denotes the spectrum of this commutative \CA, $q_{\mathsf{B}}$ denotes this particular quotient map. Now $p$ is obtained as follows: as $q_{\HH}\comp\pi$ has commutative target, it factors through $C(Gr(\widehat{\GG}))$ and $p\comp q_{\GG}=q_{\HH}\comp\pi$. Similarly, we obtain $b$ as the map completing the factorization of $q_{\mathsf{B}}\comp\beta$ through $q_{\GG}$ and $\tilde{b}$ completes the factorization of $q_{\mathsf{B}}\comp\tilde{\beta}$ through $q_{\HH}$.
\begin{theorem}\label{thm:criterion}
If $\HH$ is the Hopf image of the map $\beta$, then the Hopf image of $b$ contains $Gr(\widehat{\HH})$. In other words, the Gelfand dual $\hat{b}\colon \sigma(\mathsf{B})\to Gr(\widehat{\GG})$ satisfies $\overline{\langle{\hat{b}[\sigma(\mathsf{B})]}\rangle}\subseteq Gr(\widehat{\HH})$. 
\end{theorem}
With slight abuse of notation, this should be understood as $\overline{\langle\sigma(\mathsf{B})\rangle}\subset\HH\subset\GG$, where  $\HH$ is the Hopf image of $\beta$. As a general motto, this means that the farther from being simple the \CA\ $\mathsf{B}$ is, the better lower bound on $\HH$ we obtain.
\begin{proof}
 That $\sigma(\mathsf{B})\subseteq Gr(\widehat{\HH})$ follows from the commutativity of the above diagram and hence $\overline{\langle \sigma(\mathsf{B})\rangle}\subseteq Gr(\widehat{\HH})$, as the latter is closed subgroup of $Gr(\widehat{\GG})$.
\end{proof}
It is clear that the inclusion of \autoref{thm:criterion} can be proper, as the example of $\Lambda_{\widehat{\mathbb{F}_2}}\colon C^*(\mathbb{F}_2)\to C^*_r(\mathbb{F}_2)$ shows. Moreover, the result is not formulated in the optimal way, as one could replace $\mathsf{B}$ with the image of $\beta$ (cf. \cite[Section 2.2]{JKS16}), and the smaller subalgebra is more likely to have characters, as the example $\underline{a}\mapsto\mathrm{diag}(\underline{a})\colon c_0\hookrightarrow\mathsf{K}(\ell^2)$ shows. We will later see that even restricting the attention to the best-behaved case of coamenable compact quantum groups of Kac type, with finitely many characters, and the generating set coming from two proper subgroups (in the spirit of \cite{BCV}), the inclusion cannot be reversed.

\subsection{The quantum permutation group \texorpdfstring{$S_n^+$}{Sn+} and quantum increasing sequences}\label{sec:qpg-qis}
Quantum permutation groups $S_n^+$ were introduced in \cite{Wang98} (cf. \cite[Section 3]{SS16}). Consider the universal $C^*$-algebra generated by $n^2$-elements $u_{i,j}$, $1\leq i,j\leq n$ subject to the following relations:
\begin{enumerate}
 \item the generators $u_{i,j}$ are all projections.
 \item $\sum_{i=1}^nu_{i,j}=\mathds{1}=\sum_{j=1}^nu_{i,j}$.
\end{enumerate}
This $C^*$-algebra will be denoted $C^u(S_n^+)$. The matrix $U=[u_{i,j}]_{1\leq i,j\leq n}$ is a fundamental corepresentation of $C^u(S_n^+)$, this gives all the quantum group-theoretic data. Moreover, $S_n^+=S_n$ for $n\leq 3$ and $S_n^+\supsetneq S_n$ for $n\geq4$ and $S_n^+$ is coamenable only if $n\leq4$ (\cite{Ban99}).

The algebra of continuous functions on the set of quantum increasing sequences was defined by Curran in \cite[Definition 2.1]{Cur11}. Let $k\leq n\in\mathbb{N}$ and let $C(I^+_{k,n})$ be the universal $C^{\ast}$-algebra generated by $p_{i,j}$, $1\leq i\leq n$, $1\leq j\leq k$ subject to the following relations:
\begin{enumerate}
 \item the generators $p_{i,j}$ are all projections.
 \item each column of the rectangular matrix $P=[p_{i,j}]$ forms a partition of unity: \(\sum_{i=1}^n p_{i,j}=\mathds{1}\) for each \(1\leq j\leq k\).
 \item increasing sequence condition: \(p_{i,j}p_{i'j'}=0\) whenever $j<j'$ and $i\geq i'$.
\end{enumerate}
This definition is obtained by the liberalization philosophy (see \cite{BS09}): if one denotes by $I_{k,n}$ the set of increasing sequences of length $k$ and values in $\{1,\ldots,n\}$, then it is possible to write a matrix representation: to an increasing sequence $\underline{i}=(i_1<\ldots<i_k)$ one associates its matrix representation $A(\underline{i})\in M_{n\times k}(\{0,1\})$ as follows: $A(\underline{i})_{i_l,l}=1$ and all other entries are set to be $0$. One can check that the space of continuous functions on these matrices $C(\{A(\underline{i}):\underline{i}\in I_{k,n}\})$ is generated by the coordinate functions $x_{i,j}$ subject to the relations introduced above \textbf{and} the commutation relation (cf. the discussion after \cite[Remark 2.2]{Cur11}).

Curran defined also a ${}^{\ast}$-homomorphism $\beta_{k,n}\colon C(S_n^+)\to C(I^+_{k,n})$ (\cite[Proposition 2.5]{Cur11}) by:
\begin{itemize}
 \item $u_{i,j}\mapsto p_{i,j}$ for $1\leq i \leq n$, $1\leq j\leq k$,
 \item $u_{i,k+m}\mapsto 0$ for $1\leq m\leq n-k$ and $i<m$ or $i>m+k$,
 \item for $1\leq m\leq n-k$ and $0\leq p\leq k$, \[ u_{m+p,k+m}\mapsto \sum_{i=0}^{m+p-1} p_{i,p}-p_{i+1,p+1},\]
 where we set $p_{0,0}=\mathds{1}$, $p_{0,i}=p_{0,i}=p_{i,k+1}=0$ for $i\geq1$.
\end{itemize}
This ${}^{\ast}$-homomorphism is well defined thanks to \cite[Proposition 2.4]{Cur11}, where some new relations were identified, and the universal property of $C(S_n^+)$. $\beta_{k,n}$ are defined in such a way that when applied to the commutative $C^*$-algebras $C(S_n)\to C(I_{k,n})$ (which satisfy the same relations plus commutativity), it is precisely the ``completing an increasing sequence to a permutation'' map. More precisely, one draws the diagram of an increasing sequence $\underline{i}=(i_1<\ldots<i_k)$ in the following way: drawing $k$ dots in one row and additional $n$ dots in the row below, one connects $l$-th dot in the upper row to the $i_l$-th dot in the lower row. Then one draws additional $n-k$ dots in the upper row next to previously drawn $k$ dots and connects them as follows: $(k+j)$-th dot is connected to the $j$-th leftmost non-connected dot in the bottom row. Finally, one obtains the diagram of a permutation on $n$ letters, which is then called $\beta_{k,n}(\underline{i})$ (for the version of $\beta_{k,n}$ as a map between appropriate commutative $C^*$-algebras). 
\begin{fact}\label{fact}
$\langle I_{k,n}\rangle=S_n$ for all $n$ and all $k\neq0,n$, where $I_{k,n}\subseteq S_n$ is seen via the above map.  
\end{fact}

\begin{proposition}\label{prop:criterionapplied}
 Let $\mathbb{H}\subseteq S_n^+$ be the Hopf image of the map $\beta_{k,n}\colon C(S_4^+)\to C(I^+_{k,n})$ for $k\neq0,n$. Then $S_n\subseteq \mathbb{H}\subseteq S_n^+$
\end{proposition}
\begin{proof}
 The abelianization of $C(I^+_{k,n})$ is $C(I_{k,n})$ and the map $\beta_{k,n}$ on the level of abelianizations is the canonical map, as noted above. We conclude by \autoref{thm:criterion} together with \autoref{fact}.
\end{proof}

In what follows, we restrict our attention to the case $n=4$, $k=2$.
\begin{theorem}\label{thm:hopfimage}
 The Hopf image of the map $\beta_{2,4}\colon C(S_4^+)\to C(I^+_{2,4})$ is the whole $S_4^+$.
\end{theorem}
\begin{proof}
 Form \autoref{prop:criterionapplied} we see that the group of characters of $\mathbb{H}$, the Hopf image of $\beta$, is the permutation group $Gr(\widehat{\mathbb{H}})=S_4$. In particular, $\mathbb{H}$ contains the diagonal Klein subgroup, so is one of the groups listed in \cite[Theorem 6.1]{BB09}. It is easy to check that the group of characters of subgroups contained in \cite[Theorem 6.1]{BB09} are equal to $S_4$ only for the following two groups: $S_4$ and $S_4^+$. On the other hand, in \cite[Proposition 7.4]{SS16} it was shown that $C(I^+_{2,4})\cong (\mathbb{C}^2\ast\mathbb{C}^2)\oplus\mathbb{C}^2$ (the free product is amalgamated over $\mathbb{C}\mathds{1}$) is infinite dimensional, hence $\mathbb{H}\neq S_4$. Consequently, $\mathbb{H}=S_4^+$ is the only possibility left.
\end{proof}

\section{Group-theoretical properties of $SO_{-1}(3)$}\label{sec:gtprop}
\subsection{The \CA s $C(SO_{-1}(3))$, $C(SO(3))$ and $C(O_{-1}(2))$.}\label{sec:objects}
Let us now introduce the main players of this manuscript.
\begin{df}\label{def:SO_(-1)(3)}
The \CA\ of continuous functions on a compact quantum group $SO_{-1}(3)$ is the universal \CA\ generated by $a_{i,j}$, $1\leq i,j\leq 3$ subject to the following relations:
\begin{enumerate}
 \item The matrix $A=(a_{i,j})_{1\leq i,j\leq 3}\in M_3(C(SO_{-1}(3))$ is orthogonal, i.e.~$AA^{\top}=A^{\top}A=\mathds{1}\in M_3(C(SO_{-1}(3)))$. In particular, the generators $a_{i,j}$ are self-adjoint.
 \item $a_{i,j}a_{i,k}=-a_{i,k}a_{i,j}$ for $k\neq j$.
 \item $a_{i,j}a_{k,j}=-a_{k,j}a_{i,j}$ for $k\neq i$. 
 \item $a_{i,j}a_{k,l}=a_{k,l}a_{i,j}$ for $i\neq k$, $j\neq l$.
 \item $\sum_{\sigma\in S_3} a_{1,\sigma(1)}a_{2,\sigma(2)}a_{3,\sigma(3)}=\mathds{1}$
\end{enumerate}
$A$ is the fundamental corepresentation of $C(SO_{-1}(3))$: this defines the quantum group structure.
\end{df}

In the same spirit we can define the \CA\ of continuous functions on $SO(3)$ as the universal \CA\ generated by $x_{i,j}$, $1\leq i,j\leq 3$ subject to the following relations:
\begin{enumerate}
 \item The matrix $X=(x_{i,j})_{1\leq i,j\leq 3}\in M_3(C(SO(3))$ is orthogonal, i.e.~$AA^{\top}=A^{\top}A=\mathds{1}\in M_3(C(SO(3)))$. In particular, the generators $x_{i,j}$ are self-adjoint.
 \item $x_{i,j}x_{k,l}=x_{k,l}x_{i,j}$ for all $1\leq i,j,k,l\leq3$.
 \item $\sum_{\sigma\in S_3} \sgn(\sigma)a_{1,\sigma(1)}a_{2,\sigma(2)}a_{3,\sigma(3)}=\mathds{1}$
\end{enumerate}
It is routine to conclude from Stone-Weierstrass theorem that $C(SO(3))$ is indeed the \CA\ of continuous functions on $SO(3)$. The matrix multiplication in $SO(3)$ is encoded by $X$ being a fundamental corepresentation.

We will also need the \CA\ of continuous functions on a compact quantum group $O_{-1}(2)$. 

\begin{df}\label{def:O_(-1)(2)}
The \CA\ of continuous functions on a compact quantum group $O_{-1}(2)$ is the universal \CA\ generated by $\tilde{a}_{i,j}$, $1\leq i,j\leq 2$ subject to the relations (1-4) of \autoref{def:SO_(-1)(3)}, \emph{mutati mutandis}. As previously, the matrix $\tilde{A}=[\tilde{a}_{i,j}]_{1\leq i,j\leq2}\in M_2(C(O_{-1}(2)))$ is a fundamental corepresentation of $C(O_{-1}(2)))$. 
\end{df}

The following map yields a surjective ${}^{\ast}$-homomorphism interpreted as $O_{-1}(2)\subset SO_{-1}(3)$:
\begin{equation}\label{eq:canonicalembedding} a_{i,j}\mapsto
\left\{
\begin{array}{ll} 
\tilde{a}_{i,j} & \mathrm{ for\ } 1\leq i,j\leq2\\
\tilde{a}_{1,1}\tilde{a}_{2,2}+\tilde{a}_{1,2}\tilde{a}_{2,1} &  \mathrm{ for\ } i=j=3\\
0 & \mathrm{otherwise}
\end{array}\right.\colon C(SO_{-1}(3))\twoheadrightarrow C(O_{-1}(2))
\end{equation}
But there are more embeddings $O_{-1}(2)\subset SO_{-1}(3)$. In order to classify all of them, let us recall that these quantum groups can be described as cocycle-twists of their classical versions. Let us remark that (as will later become clear) these quantum groups are coamenable, thus \eqref{eq:canonicalembedding} gives the proper description of the notion of subgroup.

\subsection{Twistings. General Theory.}\label{sec:twistingsgeneral}  In what follows, we briefly discuss the twisting procedure and introduce the notation. We stick to the theory of \HA s, although the procedure works well for general Hopf algebras over any field.

Let $H$ be a \HA\ with coproduct $\Delta$. Recall that the algebra $H\otimes H$ can be given the \HA\ structure: the coproduct is $\Delta_2=(\id\otimes \Sigma\otimes \id)\comp(\Delta\otimes\Delta)$, where $\Sigma$ denotes the flip map. We will use the Sweedler-Heyneman notation: $\Delta(x)=x_{(1)}\otimes x_{(2)}$. A linear map $\sigma\colon H\otimes H\to\mathbb{C}$ is called a 2-cocycle if:
\begin{enumerate}
 \item it is convolution invertible: the neutral element of convolution is $m_{\mathbb{C}}\comp(\varepsilon\otimes\varepsilon)$, the convolution of $\sigma, \sigma'\colon H\otimes H\to\mathbb{C}$ is given by $\sigma\ast\sigma'=m_{\mathbb{C}}\comp(\sigma\otimes\sigma')\comp\Delta_2$,
 \item it satisfies the cocycle identity: \begin{equation}\label{eq:cocycle}\sigma(x_{(1)},y_{(1)})\sigma(x_{(2)}y_{(2)},z)=\sigma(y_{(1)},z_{(1)})\sigma(x,y_{(2)}z_{(2)})\end{equation} and $\sigma(x,1)=\varepsilon(x)=\sigma(1,x)$ for $x,y,z\in H$.
\end{enumerate}
Here and in what follows, $m_W\colon W\otimes W\to W$, for a given algebra $W$, is the multiplication map $W\otimes W\ni x\otimes y\xmapsto{m_W} x\cdot y\in W$.

Following \cite{Doi93, Sch96, BB09}, a 2-cocycle $\sigma$ provides a new \HA\ $H^{\sigma}$. As a coalgebra, $H^{\sigma}=H$, whereas the product of $H^{\sigma}$ is defined as

\[[x][y]=\sigma(x_{(1)},y_{(1)})\sigma^{-1}(x_{(3)},y_{(3)})[x_{(2)}y_{(2)}],\]
where an element $x\in H$ is denoted $[x]$ when viewed as an element of $H^{\sigma}$. In other words, $m_{H^{\sigma}}=(\sigma\otimes m_H\otimes\sigma^{-1})\comp\Delta_2^2$. The antipode of $H^{\sigma}$ can be expressed via the following formula:
\[S^{\sigma}([x])=\sigma(x_{(1)},S(x_{(2)}))\Sigma^{-1}(S(x_{(4)}),x_{(5)})[S(x_{(3)})].\]

The Hopf algebras $H$ and $H^{\sigma}$ have equivalent tensor categories of comodules (\cite{Sch96}). In our considerations we are interested in the case when the 2-cocycle is induced from a \HA\ quotient (quantum subgroup). Let $\pi\colon H\to K$ be a Hopf surjection and let $\sigma\colon K\otimes K\to\mathbb{C}$ be a 2-cocycle on $K$. Then $\sigma_{\pi}=\sigma\comp(\pi\otimes\pi)\colon H\otimes H\to\mathbb{C}$ is a 2-cocycle.

\begin{proposition}[{\cite[Lemma 4.3]{BB09}}]\label{prop:bijection} Let $\pi\colon H\to K$ be a Hopf surjection and let $\sigma\colon K\otimes K\to\mathbb{C}$ be a 2-cocycle. Then there is a bijection between:
\begin{enumerate}
 \item Hopf surjections $f\colon H\to L$ such that there exists a Hopf surjection $g\colon L\to K$ satisfying $g\comp f=\pi$, and
 \item Hopf surjections $\tilde{f}\colon H^{\sigma_\pi}\to \tilde{L}$ such that there exists a Hopf surjection $\tilde{g}\colon \tilde{L}\to K^{\sigma}$ satisfying $\tilde{g}\comp \tilde{f}=[\pi(\cdot)]$.
\end{enumerate}
The bijection is given by $\tilde{f}(\cdot)=[f(\cdot)]$.
\end{proposition}
\subsection{Characteristic subgroups.}\label{sec:characteristic} Let $\GG$ be a compact quantum group and let $\HH$ be its subgroup: let $\pi\colon C^u(\GG)\to C^u(\HH)$ be a quotient map intertwining the respective coproducts. 

\begin{df}
 We will say that $\HH$ is a \emph{characteristic subgroup} of $\GG$ if for any automorphism of $\GG$ (i.e.~a Hopf ${}^{\ast}$-homomorphism $\theta\colon C^u(\GG)\to C^u(\GG)$, cf. \cite[Section 3]{Pat13}), $\HH$ is mapped onto $\HH$ (i.e.~$\theta(\ker(\pi))=\ker(\pi)$, or in other words, there exists an automorphism $\chi\colon C^u(\HH)\to C^u(\HH)$ such that $\pi\comp\theta=\chi\comp\pi$).
\end{df}
Clearly, this notion can be described equivalently in terms of the underlying \HA, we will use this further without mentioning. An example of a characteristic subgroup is as follows.
\begin{proposition}
 The intrinsic subgroup $Gr(\widehat{\GG})$ of $\GG$ is characteristic.
\end{proposition}
\begin{proof}
 Let $\theta\colon C^u(\GG)\to C^u(\GG)$ be an automorphism of $\GG$. As the kernel of the quotient map $q\colon C^u(\GG)\to C(Gr(\widehat{\GG}))$ is an ideal generated by commutators, and as $\theta([x,y])=[\theta(x),\theta(y)]$, $\theta(\ker(q_{\GG}))\subseteq \ker(q_{\GG})$. The other inclusion follows by applying $\theta^{-1}$.
\end{proof}
There is another, more concrete, example of a characteristic subgroup, and we will use it in what follows. Let $H'=\Pol(SO_{-1}(3))$ denote the unique dense \HA\ of the quantum group $SO_{-1}(3)$, let $K=\mathbb{C}[\mathbb{Z}_2\times\mathbb{Z}_2]$ denote the group algebra of the Klein group $\mathbb{Z}_2\times\mathbb{Z}_2=\langle t_1,t_2|t_1^2=t_2^2=1, t_1t_2=t_2t_1\rangle$. Let us also denote $t_3=t_1t_2\in \mathbb{Z}_2\times\mathbb{Z}_2$ and the neutral element $t_0\in\mathbb{Z}_2\times\mathbb{Z}_2$. The Klein group can be embedded into $SO_{-1}(3)$ as follows:
\begin{equation}\label{eq:pid'}H'\ni a_{i,j}\xmapsto{\pi'_d}\delta_{i,j}t_i\in K\end{equation}
There are other occurrences of the Klein group as a subgroup of $SO_{-1}(3)$: this one will be called diagonal. Let $\pi\colon H'\to K$ be a Klein subgroup in $SO_{-1}(3)$ and consider the following factorization:

\begin{center}\begin{tikzpicture}
[bend angle=36,scale=2,auto,
pre/.style={<<-,shorten <=1pt,semithick},
post/.style={->>,shorten >=1pt,semithick}]
\node (H) at (-1.5,0) {$H'$};
\node (Hab) at (0,0) {$H'_{ab}$}
edge [pre] node[auto,swap] {$ab$} (H);
\node (Klein) at (1.5,0) {$K$}
edge [pre] node[auto,swap] {$\theta$} (Hab);
\draw[->] (H) .. controls (-0.5,0.5) and (0.5,0.5) .. node[midway] {$\pi$} (Klein) ;
 \end{tikzpicture}\end{center}
In the above diagram, $H'_{ab}$ denotes the the abelianization of $H'$: the \HA\ quotient of $H'$ by the commutator ideal, $ab$ denotes this quotient map. 

It is clear that all quotients $\pi$ onto the group algebra of the Klein group enjoy the above factorization. Let us describe it more explicitly.

\begin{lemma}\label{lem:maximalclassical}
 $H^{ab}$ is precisely the \HA\ $C(S_4)$, and the map $ab$ is given as follows: consider the canonical representation $\rho\colon S^4\to O(4)$ and consider the restriction to the subspace $(1,1,1,1)^{\perp}$: this gives an embedding $\rho\colon S_4\to O(3)$, $ab\colon H'\to C(S_4)$ acts as $a_{i,j}\xmapsto{ab} x_{i,j}\circ\rho$.
\end{lemma}
\begin{proof}
 Straightforward computation.
\end{proof}
Thus any Klein subgroup of $SO_{-1}(3)$ is a Klein subgroup in $S_4$; there are two types of Klein groups embedded into $S_4$: the easy ones, of the form: $\{\id, (1i), (kl), (1i)(kl)\}$ and the diagonal one, of the form $\{\id, (12)(34), (13)(24), (14)(23)\}$. 

\begin{remark}\label{rmk:rhoklein}
The diagonal Klein subgroup, in the above map, consists of the matrices \[\left\{\mathds{1}_{M_3}, \begin{pmatrix}-1 & 0 & 0\\ 0 & 1 & 0\\ 0 & 0 & -1\end{pmatrix}, \begin{pmatrix}1 & 0 & 0\\ 0 & -1 & 0\\ 0 & 0 & -1\end{pmatrix}, \begin{pmatrix}-1 & 0 & 0\\ 0 & -1 & 0\\ 0 & 0 & 1\end{pmatrix}\right\}.\] 
\end{remark}

\begin{lemma}\label{lem:characteristicklein} The diagonal Klein subgroup of $SO_{-1}(3)$ is a characteristic subgroup.\end{lemma}
\begin{proof}
 Any automorphism of $SO_{-1}(3)$ restricts to an automorphism of $Gr(\widehat{SO_{-1}(3)})=S_4$ and any occurrence of a Klein subgroup in $SO_{-1}(3)$ appears as a Klein subgroup of $S_4$, so it suffices to check that the diagonal (in $SO_{-1}(3)$) Klein subgroup of $SO_{-1}(3)$ is precisely the diagonal (in $S_4$) Klein subgroup of $S_4$ and that the latter is characteristic in $S_4$. Both assertions are obvious.
\end{proof}
Just to complete the picture, let us elucidate the easy Klein subgroups, providing a non-example of a characteristic subgroup.
\begin{lemma}
 All the easy Klein subgroups of $S_4$ are conjugate; the corresponding automorphism of $S_4$ extends to $SO_{-1}(3)$.
\end{lemma}
\begin{proof}
 Let $\{\id, (12), (34), (12)(34)\}$ and $\{\id, (1i), (2j), (1i)(2j)\}$ be two distinct Klein subgroups of $S_4$. It is easy to check that conjugation by $(2i)$ gives the first part of the lemma. In order to get the automorphism $u\colon H'\to H'$ extending it, simply consider the map $A\mapsto \rho(2i)A\rho(2i)$, where $\rho$ is the map from Lemma \ref{lem:maximalclassical}.
\end{proof}

\subsection{Twistings applied to $SO(3)$.}\label{sec:twistingapplied} 
Let $H=\Pol(SO(3))$ denote the unique dense \HA\ in $C(SO(3))$, recall that $K=\mathbb{C}[\mathbb{Z}_2\times\mathbb{Z}_2]$. The restriction of functions on $SO(3)$ to its diagonal subgroup gives a \HA\ surjection \[H\ni x_{i,j}\xmapsto{\pi_d}\delta_{i,j}t_i\in K.\]

Let $\sigma\colon K\otimes K\cong\mathbb{C}[(\mathbb{Z}_2\times\mathbb{Z}_2)^2]\to\mathbb{C}$ be the unique linear extension of the mapping

\begin{equation}\label{eq:diagonalcocycle} \sigma(t_i,t_j)=
\left\{
\begin{array}{rl} 
-1 & \mathrm{ for\ } (i,j)\in\{(1,1),(1,3),(2,1),(2,2),(3,2),(3,3)\}\\
1 & \mathrm{otherwise}
\end{array}\right.
\end{equation}
In other words, for $1\leq i,j\leq2$ we have that $\sigma(t_i,t_j)=-1$ if and only if $i\leq j$ and we extend this definition by bimultiplicativity. Then $\sigma$ is a 2-cocycle in the sense of \eqref{eq:cocycle}. We will work with the cocycle $\sigma_d=\sigma\circ(\pi_d\otimes\pi_d)$ on $H$. Note that $\sigma_d^{-1}=\sigma_d$.

Similarly, one can define the 2-cocycle $\sigma'_{d}\colon H'\otimes H'\to\mathbb{C}$ via $\pi_d'\colon H'\to K$ (recall the definition of $\pi_d'$ given in \eqref{eq:pid'}).
\begin{theorem}[{\cite[Theorem 5.1]{BB09}}] The \HA s $H^{\sigma_d}$ and $\Pol(SO_{-1}(3))$ are isomorphic. The isomorphism is given by $[x_{i,j}]\mapsto a_{i,j}$.
\end{theorem}
With this in hand, and the results of \autoref{sec:twistingsgeneral} and \autoref{sec:characteristic}, we are able to classify all the automorphisms of $SO_{-1}(3)$. Consider an automorphism $\theta\colon C(SO_{-1}(3))\to C(SO_{-1}(3))$ and the following diagram:

\begin{center}\begin{tikzpicture}
[bend angle=36,scale=2,auto,
pre/.style={<<-,shorten <=1pt,semithick},
post/.style={->>,shorten >=1pt,semithick}]
\node (G1) at (-1.5,0) {$H'$};
\node (G2) at (0,0) {$H'$}
edge [pre] node[auto,swap] {$\theta$} (H);
\node (Klein) at (1.5,0) {$K$}
edge [pre] node[auto,swap] {$\pi_d'$} (G2);
\draw[->] (H) .. controls (-0.5,0.5) and (0.5,0.5) .. node[midway] {$\chi\comp\pi_d'$} (Klein) ;
 \end{tikzpicture}\end{center}
Thanks to Lemma \ref{lem:characteristicklein}, the above diagram is well defined (the diagonal Klein subgroup is characteristic, hence $\pi'_d\comp\theta=\chi\comp\pi'_d$ for some automorphism $\chi\colon K\to K$), so we can use \autoref{prop:bijection} to ``untwist'' this diagram and obtain an automorphism of $SO(3)$ (which should be easier to classify). Apply the cocycle $\sigma_d'$ (recall that the Klein groups have no nontrivial twist, cf. \cite[Lemma 6.2]{BB09}) and \autoref{prop:bijection} gives us the following diagram:
\begin{center}\begin{tikzpicture}
[bend angle=36,scale=2,auto,
pre/.style={<<-,shorten <=1pt,semithick},
post/.style={->>,shorten >=1pt,semithick}]
\node (G1) at (-1.5,0) {$H$};
\node (G2) at (0,0) {$H$}
edge [pre] node[auto,swap] {$\theta^{\sigma'_d}$} (H);
\node (Klein) at (1.5,0) {$K$}
edge [pre] node[auto,swap] {$\pi_d$} (G2);
\draw[->] (H) .. controls (-0.5,0.5) and (0.5,0.5) .. node[midway] {$\chi\comp\pi_d$} (Klein) ;
 \end{tikzpicture}\end{center}
 and $\theta^{\sigma'_d}=[\theta]$, where $[\cdot]$ is understood as in \autoref{sec:twistingsgeneral}. As any automorphism of $SO(3)$ is inner, it is enough to check which of them preserve the diagonal Klein subgroup. It is clear that conjugation by $\rho(x)$, $x\in S_4$ (where $\rho\colon S_4\to O(3)$ is introduced in Lemma \ref{lem:maximalclassical}), is such an automorphism (as the diagonal Klein subgroup is characteristic in $S_4$). Using \autoref{rmk:rhoklein} one can use brute-force computations to write a formula for $F\in SO(3)$ that conjugates the diagonal Klein subgroup in $SO(3)$ and arrive at a system of constraints saying that the matrix $F$ has to have two zero entries in each row and column (the remaining entry, because of norm 1 condition in each row and column, has to be $\pm1$). There are precisely $24=4!=|S_4|$ of such matrices, hence we arrive at the following
 \begin{theorem}\label{thm:automorphisms}
  Every automorphism of $SO_{-1}(3)$ is given by $A\mapsto \rho(x)^{\top}A\rho(x)$ for some $x\in S_4$. In other words, $Aut(SO_{-1}(3))\cong S_4$.
 \end{theorem}
\subsection{The embeddings \texorpdfstring{$O_{-1}(2)\subset SO_{-1}(3)$}{Oq(2)<SOq(3)}}\label{sec:embeddings1}
Recall an embedding $O_{-1}(2)\subset SO_{-1}(3)$ from \eqref{eq:canonicalembedding}:
\[a_{i,j}\mapsto
\left\{
\begin{array}{ll} 
\tilde{a}_{i,j} & \mathrm{ for\ } 1\leq i,j\leq2\\
\tilde{a}_{1,1}\tilde{a}_{2,2}+\tilde{a}_{1,2}\tilde{a}_{2,1} &  \mathrm{ for\ } i=j=3\\
0 & \mathrm{otherwise}
\end{array}\right.\colon C(SO_{-1}(3))\twoheadrightarrow C(O_{-1}(2))\]
There are other embeddings $O_{-1}(2)\subset SO_{-1}(3)$, which are classified by the following 
\begin{theorem}\label{thm:quantumembeddings}
 There are three copies of $O_{-1}(2)\subset SO_{-1}(3)$. The three copies are conjugate (via an automorphism described in \autoref{thm:automorphisms}).
\end{theorem}
\begin{proof} Recall from \cite[Theorem 7.1 \& Proposition 7.3]{BB09} that the group of characters of $O_{-1}(2)$ is isomorphic to $D_4$, the dihedral group of a square (or one can simply compute using relations (2-3) of \autoref{def:O_(-1)(2)} together with commutation relations: one obtains the only eight orthogonal matrices with entries $0,\pm1$ and uses classification of groups of order $8$).  Let $\Phi\colon H'=\Pol(SO_{-1}(3))\to \Pol(O_{-1}(2))$ be a \HA\ quotient. Consider the following diagram:
 \begin{center}\begin{tikzpicture}
[bend angle=36,scale=2,auto,
pre/.style={<<-,shorten <=1pt,semithick},
post/.style={->>,shorten >=1pt,semithick}]
\node (G) at (-1,0) {$H'$};
\node (H) at (1,0) {$\Pol(O_{-1}(2))$}
edge [pre] node[auto,swap] {$\Phi$} (G);
\node (Gclass) at (0,-1) {$C(S_4)$}
edge [pre] node[auto] {$q_{SO_{-1}(3)}$} (G);
\node (Hclass) at (2,-1) {$C(D_4)$}
edge [pre] node[auto,swap] {$q_{O_{-1}(2)}$} (H)
edge [pre] node[auto,swap] {$\varphi$} (Gclass);
 \end{tikzpicture}\end{center} 
The existence of the map $\varphi$ as in the diagram above follows from the universal property of abelianization. Because all the involved morphisms are \HA\ morphisms, so is $\varphi$. Similarly, because all the involved morphisms are surjections, so is $\varphi$. Thus $\hat{\varphi}$, the Gelfand transform of $\varphi$, is a monomorphism $\hat{\varphi}\colon D_4\hookrightarrow S_4$. Let us take for granted that the image of $\hat{\varphi}$ contains the diagonal Klein subgroup of $S_4$ (the proof of this statement is postponed to \autoref{lem:imagehatphi} just below the end of the proof of \autoref{thm:quantumembeddings}). 

As the diagonal Klein subgroup is characteristic in $SO_{-1}(3)$, this gives us the following diagram of morphisms:
\begin{center}\begin{tikzpicture}
[bend angle=36,scale=2,auto,
pre/.style={<<-,shorten <=1pt,semithick},
post/.style={->>,shorten >=1pt,semithick}]
\node (G) at (-1.5,0) {$H'$};
\node (H) at (0,0) {$\Pol(O_{-1}(2))$}
edge [pre] node[auto,swap] {$\Phi$} (G);
\node (Klein) at (1.5,0) {$K$}
edge [pre] node[auto,swap] {$\chi\circ\tilde{\pi}$} (H);
\draw[->] (G) .. controls (-0.5,-0.5) and (0.5,-0.5) .. node[midway] {$\pi'_d$} (Klein) ;
 \end{tikzpicture}\end{center}
 where $\tilde{\pi}$ is obtained by composing $q_{O_{-1}(2)}$ with \HA\ quotient map corresponding to restriction to the diagonal Klein subgroup in $\hat{\varphi}(D_4)$ and $\chi$ is an automorphism of the Klein group (its existence is a consequence of the diagonal Klein group being characteristic). Using \autoref{prop:bijection}, we untwist this diagram and arrive at
\begin{center}\begin{tikzpicture}
[bend angle=36,scale=2,auto,
pre/.style={<<-,shorten <=1pt,semithick},
post/.style={->>,shorten >=1pt,semithick}]
\node (G) at (-1.5,0) {$H$};
\node (H) at (0,0) {$\Pol(O(2))$}
edge [pre] node[auto,swap] {$[\Phi]$} (G);
\node (Klein) at (1.5,0) {$K$}
edge [pre] node[auto,swap] {$\Pi$} (H);
\draw[->] (G) .. controls (-0.5,-0.5) and (0.5,-0.5) .. node[midway] {$\pi_d$} (Klein) ;
 \end{tikzpicture}\end{center} 
 The closed subgroups of $SO(3)$ isomorphic to $O(2)$ are all of the form
 \[\left\{F\begin{pmatrix*} A & 0 \\ 0 & \det(A)\end{pmatrix*}F^{\top}:A\in O(2)\right\} \]
 for some matrix $F\in SO(3)$ (see, e.g. \cite[Theorem 6.1]{GSS}). The occurrence of $O(2)$ in $SO(3)$ coming from the above diagram contains the diagonal Klein subgroup. Because $\mathbb{Z}_2\times\mathbb{Z}_2\subseteq D_4$ is characteristic, we know from (the proof of) \autoref{thm:automorphisms} that the matrices $F$ are necessarily of the form $\rho(x)$ for some $x\in S_4$.  To verify \eqref{eq:canonicalembedding} it is then enough to check that 
 \[\begin{split}
 [\det(\tilde{X})]=[\tilde{x}_{1,1}\tilde{x}_{2,2}]-[\tilde{x}_{1,2}\tilde{x}_{2,1}]=\\
 =\sigma(t_1,t_2)\sigma(t_1,t_2)[\tilde{x}_{1,1}][\tilde{x}_{2,2}]+\sigma(t_1,t_2)\sigma(t_2,t_1)[\tilde{x}_{1,2}][\tilde{x}_{2,1}]=\\
   =\tilde{a}_{1,1}\tilde{a}_{2,2}+\tilde{a}_{1,2}\tilde{a}_{2,1}=\perm(\tilde{A}) 
   \end{split} \]\end{proof}
   \begin{lemma}\label{lem:imagehatphi}Image of $\hat{\varphi}$ contains the diagonal Klein subgroup of $S_4$.  \end{lemma}
\begin{proof}
Up to an inner automorphism, the only way to embed the dihedral group into the symmetric group is via \[\hat{\varphi}(D_4)=\{\id, (12), (34), (12)(34), (13)(24), (14)(23), (1234), (1432)\}\]
and the diagonal Klein subgroup of $S_4$ is precisely $\{\id, (12)(34), (13)(24), (14)(23)\}$, which is characteristic (hence it appears as a subgroup of any possible occurrences of $D_4$ in $S_4$).
\end{proof}
 In summary, \autoref{thm:quantumembeddings} says that any embedding $O_{-1}(2)\subset SO_{-1}(3)$ is given by the following map:
 \begin{equation}\label{eq:genericembedding1}
 A\xmapsto{\Phi_x} \rho(x)\begin{pmatrix*} \tilde{A} & 0 \\ 0 & \perm(\tilde{A})\end{pmatrix*}\rho(x)^{\top}
\end{equation}
where $x\in S_4$ and $\rho$ is as in \autoref{lem:maximalclassical} and $\perm(\tilde{A})=\tilde{a}_{1,1}\tilde{a}_{2,2}+\tilde{a}_{1,2}\tilde{a}_{2,1}$ is the permanent.
\subsection{Conclusions}\label{sec:conclusions}
Recall from \cite{Bic00} that $S_4^{\tau}$ is a cocycle twist of $S_4$ by the cocycle $\sigma$ induced from a non-diagonal Klein subgroup. From \cite[Theorem 6.1]{BB09} we know that there $S_4^{\tau}\subseteq S_4^+$ and that $Gr(\widehat{S_4^{\tau}})=D_4$ by \cite[Lemma 6.7]{BB09}. Let then $O_{-1}(2)\subseteq S_4^+$ be embedded in such a way that $Gr(\widehat{O_{-1}(2)})=Gr(\widehat{S_4^{\tau}})$ as subgroups of $S_4=Gr(\widehat{S_4^+})$ (we know from \autoref{thm:quantumembeddings} that it is possible to find such a copy of $O_{-1}(2)$). But \cite[Theorem 7.1]{BB09}, establishing the full list of subgroups of $O_{-1}(2)$, ensures us that $S_4^{\tau}\not\subset O_{-1}(2)$ and thus $\GG=\overline{\langle O_{-1}(2),S_4^{\tau}\rangle}=S_4^+$, as this group as strictly bigger than $O_{-1}(2)$ and there are no intermediate groups between $O_{-1}(2)$ and $S_4^+$. But at the same time, 
\[S_4=Gr(\widehat{\GG})\neq\langle D_4,D_4\rangle=D_4\]
hence the inclusion in \autoref{thm:criterion} can be proper.

An additional consequence of our considerations is the following
\begin{proposition}
$S_4^+=\overline{\langle S_4^{\tau}\cup S_4\rangle}$ 
\end{proposition}
\begin{proof}
 Let $\GG=\overline{\langle S_4^{\tau}\cup S_4\rangle}$. As $Gr(\widehat{\GG})=S_4$ and $\GG\neq S_4$ (because $S_4^{\tau}\not\subset S_4$), we can check on the list of \cite[Theorem 6.1]{BB09} that the only remaining quantum subgroup of $S_4^+$ with $Gr(\GG)=S_4$ is $S_4^+$ itself.
\end{proof}
\begin{remark}
  Recall $\widehat{\GG}$ is hyperlinear if and only if $L^{\infty}(\GG,h_{\GG})$ can be embedded into $\mathsf{R}^{\omega}$, where $\mathsf{R}$ is the hyperfinite $II_1$ factor and $\omega$ is a principal ultrafilter.
 \end{remark}
\begin{corollary}
 $\widehat{S_4^+}$ is hyperlinear.
\end{corollary}
\begin{proof}
 This follows immediately from the fact that $S_4$ and $S_4^{\tau}$ are finite (and hence their duals are hyperlinear) and \cite[Theorem 3.6]{BCV}.
\end{proof}
Let us mention here that this result can be also proven by employing the fact that $C(S_4^+)$ is nuclear ($\widehat{S_4^+}$ is amenable). However, the above proof, as the proof of \cite[Theorem 3.6]{BCV}, are much more elementary than the proof of nuclearity of $C(S_4^+)$.

\section*{Acknowledgement} The author would like to thank Piotr M.~So{\l}tan, his advisor, for patient guidance and for directing the author towards the questions studied in this note as well as for careful reading of this manuscript. The author was partially supported by the NCN (National Center of Science) grant 2015/17/B/ST1/00085.

\bibliography{embeddings}

\begin{thebibliography}{10}

\bibitem{Ban99}
{\sc T.~Banica}, {\em Symmetries of a generic coaction}, Math. Ann., 314
  (1999), pp.~763--780.

\bibitem{Ban14}
\leavevmode\vrule height 2pt depth -1.6pt width 23pt, {\em Truncation and
  duality results for {H}opf image algebras}, Bull. Pol. Acad. Sci. Math., 62
  (2014), pp.~161--180.

\bibitem{BB09}
{\sc T.~Banica and J.~Bichon}, {\em Quantum groups acting on 4 points}, J.
  Reine Angew. Math., 626 (2009), pp.~75--114.

\bibitem{BB10}
\leavevmode\vrule height 2pt depth -1.6pt width 23pt, {\em Hopf images and
  inner faithful representations}, Glasg. Math. J., 52 (2010), pp.~677--703.

\bibitem{BS09}
{\sc T.~Banica and R.~Speicher}, {\em Liberation of orthogonal {L}ie groups},
  Adv. Math., 222 (2009), pp.~1461--1501.

\bibitem{Bic00}
{\sc J.~Bichon}, {\em Quelques nouvelles d\'eformations du groupe
  sym\'etrique}, C. R. Acad. Sci. Paris S\'er. I Math., 330 (2000),
  pp.~761--764.

\bibitem{BY14}
{\sc J.~Bichon and R.~Yuncken}, {\em Quantum subgroups of the compact quantum
  group {$\rm SU_{-1}(3)$}}, Bull. Lond. Math. Soc., 46 (2014), pp.~315--328.

\bibitem{BCV}
{\sc M.~{Brannan}, B.~{Collins}, and R.~{Vergnioux}}, {\em {The Connes
  embedding property for quantum group von Neumann algebras}}, ArXiv e-prints,
  (2014).

\bibitem{Cur11}
{\sc S.~Curran}, {\em A characterization of freeness by invariance under
  quantum spreading}, J. Reine Angew. Math., 659 (2011), pp.~43--65.

\bibitem{DKSS}
{\sc M.~Daws, P.~Kasprzak, A.~Skalski, and P.~M. So{\l}tan}, {\em Closed
  quantum subgroups of locally compact quantum groups}, Adv. Math., 231 (2012),
  pp.~3473--3501.

\bibitem{Doi93}
{\sc Y.~Doi}, {\em Braided bialgebras and quadratic bialgebras}, Comm. Algebra,
  21 (1993), pp.~1731--1749.

\bibitem{GSS}
{\sc M.~Golubitsky, I.~Stewart, and D.~G. Schaeffer}, {\em Singularities and
  groups in bifurcation theory. {V}ol. {II}}, vol.~69 of Applied Mathematical
  Sciences, Springer-Verlag, New York, 1988.

\bibitem{PJphd}
{\sc P.~J{\'o}ziak}, {\em Hopf images in Locally compact quantum groups}, PhD
  thesis, IMPAN, 2016.

\bibitem{JKS16}
{\sc P.~{J{\'o}ziak}, P.~{Kasprzak}, and P.~M. {So{\l}tan}}, {\em {Hopf images
  in locally compact quantum groups}}, preprint.

\bibitem{KN13}
{\sc M.~Kalantar and M.~Neufang}, {\em From quantum groups to groups}, Canad.
  J. Math., 65 (2013), pp.~1073--1094.

\bibitem{KS09}
{\sc C.~K{\"o}stler and R.~Speicher}, {\em A noncommutative de {F}inetti
  theorem: invariance under quantum permutations is equivalent to freeness with
  amalgamation}, Comm. Math. Phys., 291 (2009), pp.~473--490.

\bibitem{KS12}
{\sc D.~Kyed and P.~M. So{\l}tan}, {\em Property ({T}) and exotic quantum group
  norms}, J. Noncommut. Geom., 6 (2012), pp.~773--800.

\bibitem{Pat13}
{\sc I.~Patri}, {\em Normal subgroups, center and inner automorphisms of
  compact quantum groups}, Internat. J. Math., 24 (2013), pp.~1350071, 37.

\bibitem{Pod95}
{\sc P.~Podle{\'s}}, {\em Symmetries of quantum spaces. {S}ubgroups and
  quotient spaces of quantum {${\rm SU}(2)$} and {${\rm SO}(3)$} groups}, Comm.
  Math. Phys., 170 (1995), pp.~1--20.

\bibitem{Sch96}
{\sc P.~Schauenburg}, {\em Hopf bi-{G}alois extensions}, Comm. Algebra, 24
  (1996), pp.~3797--3825.

\bibitem{SS16}
{\sc A.~{Skalski} and P.~M. {So{\l}tan}}, {\em {Quantum families of invertible
  maps and related problems}}, Canad. J. Math., 68 (2016), pp.~698--720.

\bibitem{Wang98}
{\sc S.~Wang}, {\em Quantum symmetry groups of finite spaces}, Comm. Math.
  Phys., 195 (1998), pp.~195--211.

\bibitem{Wor87}
{\sc S.~L. Woronowicz}, {\em Compact matrix pseudogroups}, Comm. Math. Phys.,
  111 (1987), pp.~613--665.

\bibitem{Wor95}
\leavevmode\vrule height 2pt depth -1.6pt width 23pt, {\em Compact quantum
  groups}, in Sym\'etries quantiques ({L}es {H}ouches, 1995), North-Holland,
  Amsterdam, 1998, pp.~845--884.

\end{thebibliography}
\bibliographystyle{siam}
\end{document}